\documentclass[10 pt,a4paper,notitlepage,twoside,openright]{article}

\usepackage{amsmath}
\usepackage{amsfonts}
\usepackage{amssymb}
\usepackage{layout}
\usepackage{pifont}
\usepackage{graphicx}

\usepackage{}

\newtheorem{theorem}{Theorem}[section]
\newtheorem{lemma}[theorem]{Lemma}
\newtheorem{corollary}[theorem]{Corollary}
\newtheorem{proposition}[theorem]{Proposition}

\newtheorem{remark}{Remark}[section]

\newenvironment{proof}
{\noindent{\itshape -- Proof: }}
{\hfill $\square$}

\DeclareMathOperator{\curl}{curl}

\DeclareMathOperator{\im}{im}

\DeclareMathOperator{\supp}{supp}
\DeclareMathOperator{\myspan}{span}

\DeclareMathOperator{\st}{st}

\newcommand{\rmd}{\mathrm{d}}

\newcommand{\rmC}{\mathrm{C}}

\newcommand{\rmH}{\mathrm{H}}
\newcommand{\rmL}{\mathrm{L}}

\newcommand{\bbR}{\mathbb{R}}

\newcommand{\calH}{\mathcal{H}}

\newcommand{\calP}{\mathcal{P}}

\newcommand{\calT}{\mathcal{T}}

\newcommand{\beq}{\begin{equation}}
\newcommand{\eeq}{\end{equation}}

\newcommand{\ts}{\textstyle}

\newcommand{\id}{\mathrm{id}}

\newcommand{\myand}{\textrm{ and }}

\newcommand{\for}{\textrm{ for }}
\newcommand{\myfor}{\for}

\newcommand{\mixed}[8]{
\left\{
\begin{array}{l}
#1 \in #3 \\
#2 \in #4
\end{array}
\right.
\quad
\left\{
\begin{array}{llll}
\forall #5 \in #3 \quad & #7 \\
\forall #6 \in #4 \quad & #8
\end{array}
\right.
}

\newcommand{\witharrowunder}[1]{
\begin{array}{c}
#1 \\
\longrightarrow
\end{array}
}

\newcommand{\witharrowright}[1]{
\begin{array}{cc}
#1&\downarrow
\end{array}
}

\newcommand{\bibauthor}[1]{{\scshape #1}: }
\newcommand{\bibtitle}[1]{{\itshape #1}; }
\newcommand{\bibspec}[1]{{\upshape #1}.}

\title{\vspace*{-1.5cm}
\begin{flushright}
\begin{minipage}{4cm}
\tiny
{\sc Dept. of Math. \hfill University of Oslo\\
Pure Mathematics \hfill No.~19\\
ISSN 0806--2439 \hfill August 2005}
\end{minipage}
\end{flushright}
\vskip1cm
Stability of Hodge decompositions\\ in finite element spaces of differential forms\\ in arbitrary dimension}
\author{Snorre H. {\sc Christiansen}\footnote{CMA, Universitetet i Oslo, PB 1053 Blindern, NO-0316 Oslo, Norway. email : {\tt snorrec@math.uio.no}}
}

\date{July 21, 2005}

\begin{document}

\maketitle

\begin{abstract}
We elaborate on the interpretation of some mixed finite element spaces in terms of differential forms. First we develop a framework in which we show how tools from algebraic topology can be applied to the study of their cohomological properties. The analysis applies in particular to certain $hp$ finite element spaces, extending results in trivial topology often referred to as the exact sequence property. Then we define regularization operators. Combined with the standard interpolators they enable us to prove discrete Poincar\'e-Friedrichs inequalities and discrete Rellich compactness for finite element spaces of differential forms of arbitrary degree on compact manifolds of arbitrary dimension.
\end{abstract}

\section{Introduction}
A close kinship between some mixed finite element spaces  and some old constructs of algebraic topology such as simplicial cochains  was pointed out by Bossavit \cite{Bos88} and has received an ever increasing attention among numerical analysts, see in particular Arnold \cite{Arn02}. Early papers on mixed finite elements include Raviart-Thomas \cite{RavTho77}, N\'ed\'elec \cite{Ned80}\cite{Ned86} and Brezzi-Douglas-Marini \cite{BreDouMar85}. Surveys on mixed finite elements can be found in Brezzi-Fortin \cite{BreFor91} and Roberts-Thomas \cite{RobTho91}, whereas for material on simplicial cochains we refer to Spanier \cite{Spa66} and  Gelfand-Manin \cite{GelMan03}. This link has led to a reinterpretation of many results in numerical analysis, see for instance  Hiptmair \cite{Hip99} and Boffi \cite{Bof01}, and has inspired the construction of new finite element spaces, see for instance Arnold-Winther \cite{ArnWin02} and Buffa-Christiansen \cite{BufChr05}.

More specifically, some interpolation operators are known to provide a commuting diagram, linking the De Rham complex of smooth differential forms to complexes of finite element spaces of piecewise polynomial forms satisfying compatibility conditions on interfaces between the cells of the mesh. Based on such commuting diagrams it is often remarked that when the De Rham sequence is exact the finite element sequence is also exact. To carry out this reasoning one usually extends the interpolators from smooth forms to spaces containing also finite element forms, which leads to some technical difficulties. For instance one often considers Sobolev spaces of differential forms. But continuity of interpolators on such spaces is based on the Sobolev injection theorems (see  e.g. Adams-Fournier \cite{AdaFou03}), which in higher dimensions require $W^{sp}$ spaces with $s$ and $p$ large. Such norm estimates are also used when the commuting diagram property is used to prove convergence estimates such as discrete compactness in the sense of Kikuchi \cite{Kik89} see Boffi \cite{Bof00}\cite{Bof01}. This property is crucial to the study of eigenvector approximations, see Boffi et al. \cite{BofFerGasPer99} and Caorsi-Fernandes-Raffetto \cite{CaoFerRaf01}. The lack of continuity of the standard interpolators leads to technical difficulties even in two and three space dimensions. For many problems in physics it would be desirable to have estimates in four space dimensions and dimensions higher than four arise for instance in applications to finance.% In this introduction I will first comment on the results of this paper regarding cohomological properties such as the exact sequence property, then on the results concerning the obtention of norm estimates, for which we develop an appropriate regularization operator. 

%\paragraph{Cohomological properties.}

The argument using commuting diagrams to prove that the discrete complex has the same cohomology as the continuous one, is limited to trivial topologies (trivial De Rham cohomology groups, except perhaps the first or last one). But for many problems in electromagnetics the topology of the device is non trivial and this is an essential feature of the way it operates. For instance in a transformer the electric currents of the coils excite an electric current on an iron torus, mainly along a certain harmonic vector field.  The space of harmonic vector fields on a compact manifold can be seen as a realization of a cohomology group of (classes of) $1$-forms via a Riemannian metric -- this is part of Hodge theory, see e.g. Taylor \cite{Tay96} (Chapter 5). Thus for the numerical simulation of such devices it is a minimum requirement that the cohomological properties of finite element spaces are correct.  Devices with non-trivial topology are also the object of recent numerical studies, see in particular Rapetti-Dubois-Bossavit \cite{RapDubBos03}. However these authors restrict attention to lowest order finite elements and domains in $\bbR^3$ with connected boundaries. Partial results on the dimensions of various finite element cohomology groups have long been obtained using Euler-Poincar\'e formulas, see for instance N\'ed\'elec \cite{Ned82}. 

That the De Rham sequence on general manifolds has the same cohomology as lowest order finite elements is in fact a deep theorem of De Rham, since the lowest order finite elements correspond to standard simplicial cochains. This by now rather well-known correspondence between finite elements and simplicial cochains was pointed out by Bossavit \cite{Bos88} following a remark by Kotiuga and is detailed for completeness in Proposition \ref{prop:fe} of this paper. For a proof of De Rham's theorem we refer to Weil \cite{Wei52}.
%; a textbook presentation of this proof in the context of spectral sequences can be found in Bott-Tu \cite{BotTu82} (Chapter II).
In this paper we provide a sufficient condition under which other finite element spaces also have cohomology naturally isomorphic to the De Rham cohomology.  The analysis applies in particular to certain $hp$ finite element spaces of differential forms, whose efficiency for simulating electromagnetic problems has been demonstrated by Demkowicz et al., see the review \cite{Dem03}.

In fact we will work with arbitrary simplicial complexes which need not be manifolds, so we shall rather prove that other finite element spaces have cohomology isomorphic to simplicial cohomology. This generalization of the kind of topological space on which we construct functional spaces, is crucial to the method of proof. Simplicial complexes which do not generate manifolds also appear in applications: for instance, returning to the example of electromagnetic phenomena, antennas are often constructed by adjoining more than two metal sheets (called screens in this setting) along edges. Such antennas are also successfully simulated in many industrial codes, though there is a lack of theory for the error analysis of these simulations. This paper can also be seen as a contribution to a construction of a framework where such problems can be addressed. For instance there could be interest in extending the results of Buffa - Christiansen \cite{BufChr03} to the case of several interconnected screens. We therefore provide a detailed description of piecewise smooth differentiable forms on arbitrary simplicial complexes.

We notice furthermore that a proof of the fact that certain $hp$ finite element spaces have the right cohomology can be found in Hiptmair \cite{Hip02} (Theorem 3.7 p. 273). Apart from greater generality, the main advantage of our result lies perhaps in the simplicity of the hypothesis and the efficiency of the proof, taking full advantage of the language and tools of homological algebra. Indeed the present results turn out to be a rather straightforward application of the basic tools of homological algebra and, if nothing else, this part of the paper might serve as an illustration of how these tools can be applied in a finite element setting. Probably the results of \S \ref{sec:def} and \S \ref{sec:cohom} will not come as a surprise to the reader, but the framework is perhaps at variance with those usually adopted by both algebraic topologists and numerical analysts. 

%\paragraph{Regularization.}
That the dimension of cohomology groups of various finite element space is the ``right'' one, can be seen as an algebraic stability property. But the analysis of approximation of eigenvalues, of say the $\curl \curl$ operator relevant to electromagnetics, also requires metric properties such as the discrete compactness property already referred to. In order to obtain such estimates it has been argued that one should look for  interpolation operators which are projectors commuting with the exterior derivative, with enough continuity properties in terms of Sobolev norms. As already mentioned the standard interpolators lack suitable continuity. On the other hand so-called Cl\'ement interpolation \cite{Cle75} is defined on rough functions but is not inserted in a suitable commuting diagram. Variants designed to remedy on this have been designed, for instance in Bernardi-Girault \cite{BerGir98} and Girault-Scott \cite{GirSco03} but were not found to be directly transposable to our problem (though some of the techniques of proof we use are inspired by these papers). In this paper we construct other regularization operators. Unfortunately they do not commute with the exterior derivative and they do not leave the finite elements spaces invariant. However we can control, through an additional parameter, the amount by which these two properties fail, in appropriate norms. Composing with the standard interpolators gives us a tool sufficiently powerful for the goals we set for ourselves concerning stability and compactness properties of Hodge decompositions.

%Our variant of De Rham regularization (chapter III, \S 15) consists in pulling back regularization by convolution on $\bbR^n$ to a neighborhood of $|T|$ in $| \st (T)|$ for each $T \in \calT^n$, and then composing all these local regularizations, to obtain global regularization operators $\Omega^k_\calT(M)  \to \Omega^k(M)$. The point is that the regularization thus obtained is homotopic to the identity on $\Omega^k_\calT(M)$, in particular it commutes with the exterior derivative.

%We would like to obtain an estimate on the distance of elements of $\calH^k_\calT(M)$ to the space $\calH^k(M)$. Such estimates enter estimates on so-called discrete Hodge decompositions, and enable one to prove so-called discrete compactness results in the sense of Kikuchi \cite{Kik89} (see also Boffi \cite{Bof01}), and discrete div-curl lemmas as in Christiansen \cite{Chr05}, in arbitrary topology.

\paragraph{Organization.}

The paper is organized as follows: In \S \ref{sec:def} we provide the definitions of simplicial complexes we will work with and introduce a space of piecewise smooth differential forms satisfying a compatibility condition along interfaces. For brevity the latter will be referred to as compatible forms. We define lowest order finite elements as particular finite dimensional spaces of compatible forms and check that they correspond to simplicial cochains (Proposition \ref{prop:fe}). Then, in \S \ref{sec:cohom}, we prove that the inclusion of lowest order finite elements in the spaces of compatible forms, induces isomorphisms in cohomology (Theorem \ref{theo:cohom}). We then consider complexes which are intermediate between lowest order finite elements and general smooth compatible forms. Under natural assumptions involving interpolation operators, we prove that inclusions (and interpolators) induce isomorphisms in cohomology (Proposition \ref{prop:highorder}). We also prove that the standard high order finite-elements behave well under wedge products (Proposition \ref{prop:wedge}). Next, in \S \ref{sec:regul} we restrict attention to smooth compact manifolds (without boundary). We construct a regularization operator and use it to prove a discrete version of the Poincar\'e-Friedrich inequality (Proposition \ref{prop:discfried}) and Rellich compactness (Corollary \ref{cor:disccomp}).

\section{Definitions\label{sec:def}}
\subsection{Simplicial complexes}
A \emph{simplicial complex}  is a set $\calT$ of finite non-empty sets with the property:
\begin{equation}
\forall T \in \calT \ \forall T' \subset T  \quad T' \neq \emptyset \Rightarrow T'\in \calT.
\end{equation}
All simplicial complexes considered in this paper are themselves \emph{finite}, so we shall take the liberty of not repeating this additional requirement. A classical reference on simplicial complexes is Spanier \cite{Spa66} chapter 3. A more abstract approach is exposed in Gelfand-Manin \cite{GelMan03}. The non-empty elements of $\calT$ will be called \emph{simplexes}, and a simplex with $k+1$ elements will be said to be $k$-dimensional or to be a $k$-simplex. The elements of a simplex are called \emph{vertices}. For each integer $k$ we denote by $\calT^k$ the subset of $\calT$ consisting of $k$-dimensional simplexes. The set of vertices of (of elements of) $\calT$, which is the union of $\calT$, will be frequently be identified with $\calT^0$.  If $T$ is a simplex and $T'$ a subset of $T$ we say that $T'$ is a \emph{face} of $T$ ; if $T'$ is $l$-dimensional we also call $T'$ an $l$-face of $T$.

For integer $k$ the following notation is useful:
\begin{equation}
[k]=  \{0,1, \cdots, k\}.
\end{equation}
Given a $k$-simplex $T$, on the set of bijections $[k] \to T$, the relation defined by:
\begin{equation}
\tau \sim \tau' \iff \tau^{-1}\tau' \textrm{ is an even permutation of } [k],
\end{equation}
is an equivalence relation which has two equivalence classes when $k \geq 1$. An orientation of $T$ is the choice of such an equivalence class. We suppose that for each $k$-simplex $T$ we have chosen a bijection:
\begin{equation}
\sigma_T : [k] \to T.
\end{equation}
We will equip $T$ with the orientation induced by  $\sigma_T$ (its equivalence class).

An \emph{affine realization} of a simplex $T \in \calT^k$ is an injection:
\begin{equation}
\rho_T: T \to V,
\end{equation}
into an affine space $V$ such that the range $\rho_T(T)$ has a $k$-dimensional affine span in $V$. We allow $V$ to have dimension larger than $k$. The closed convex hull of $\rho_T(T)$ in $V$ will be denoted $|T|$ and also called an affine realization of $T$ when no confusion can arise as a consequence. An affine realization of a simplicial complex $\calT$ is the data consisting of an affine realization of each of its simplexes; it is denoted $|\calT|$. If $T,T'$ are simplexes in $\calT$ such that $T' \subset T$ there is a unique affine map $|T'| \to |T|$ which coincides with $\rho^{\phantom{-1}}_{T^{\phantom{\prime}}} \!\rho_{T'}^{-1}$ on $\rho_{T'}(T')$; it is injective and we call it the \emph{canonical injection} :
\begin{equation}
i_{TT'}:|T'| \to |T|.
\end{equation}
We notice that for any $T'' \subset T' \subset T$ we have :
\begin{equation}
i_{TT}= \id_{|T|} \myand i_{TT''}=i_{TT'} \circ i_{T'T''}.
\end{equation}
In the following we consider simplicial complexes equipped with one (and only one) affine realization. 

We denote by $\Omega^k(\calT)$ the collection of all families $(u_T)_{T\in \calT}$ such that for each $T\in \calT$, $u_T$ is a smooth differential $k$-form on $|T|$ (by this we mean a $k$-form having a smooth extension to the affine space generated by $|T|$), and such that for each $T,T' \in \calT$ with the property that $T' \subset T$, we have the following compatibility condition on the pull-backs:
\begin{equation}
(i_{TT'})^\star u_T = u_{T'}.
\end{equation}
The elements of $\Omega^k(\calT)$ will be simply referred to as \emph{compatible} $k$-forms (or just compatible forms).
The exterior derivative $\rmd$ and wedge-product $\wedge$, applied component-wise to compatible forms yield compatible forms since these operations behave naturally under pull-backs, see e.g. Lang \cite{Lan95} chapter V. 

If $u=(u_T)_{T\in \calT}\in \Omega^k(\calT)$, it will be convenient to call $u_T$  the \emph{restriction} of $u$ to $|T|$ and denote it by $u|_T$. If $\calT' \subset \cal T$ is also a simplicial complex we define the restriction of $u$ to $\calT'$ by:
\begin{equation}
u|_{\calT'}= (u_T)_{T\in \calT'}\in \Omega^k(\calT').
\end{equation}
It is sometimes convenient to associate with a simplex, the simplicial complex consisting of all its subsets. For instance if $T$ is a simplex embedded in an affine space, this enables us to speak of $\Omega^k(T)$. We notice that if $u\in \Omega^k(T)$ in the above sense, then $u$ is uniquely determined by $u|_T$, and that conversely any $k$-form on $|T|$ gives rise to an element of $\Omega^k(T)$. We denote by $\partial T$ the simplicial complex consisting of all strict subsets (called \emph{proper} faces) of $T$. %The affine realization of $\partial T$ is thus the boundary of $|T|$ in the usual sense of corner manifolds.  

The following definition will be useful when we construct regularizations. For a given $T \in \calT$ we denote by $\st(T)$ the \emph{star} of $T$, defined by:
\begin{equation}\label{eq:star}
\st(T)= \{ T' \in \calT \ : \ T \cap T' \neq \emptyset \}.
\end{equation}

\subsection{Finite elements}
We now define lowest order finite elements as finite dimensional subspaces of $\Omega^k(\calT)$. This is essentially a reformulation of Bossavit \cite{Bos88} and the expressions appearing here can at least be traced back to Weil \cite{Wei52}. Weil (following DeRham) considers a good cover of a manifold and a subordinated partition of unity consisting of smooth functions. A cover gives rise to a simplicial complex: the nerve of the cover. We, on the other hand, suppose that the simplicial complex is given and use it to obtain a partition of unity consisting of compatible piecewise smooth functions%(in his paper the simplicial complex $\calT$ appears as the nerve of a simple open covering of a given manifold)
.

For each vertex $i$, denote by $\lambda_i$ the element of $\Omega^0(\calT)$ such that for each simplex $T$ having $i$ as vertex $\lambda_i|_T$ is the barycentric coordinate map on $|T|$ relative to $i$, and $\lambda_i$ is $0$ on simplices not having $i$ as vertex. In other words $\lambda_i$ is the unique piecewise affine function on $|\calT|$ such that $\lambda_i(j) = \delta_{ij}$ for all $i,j \in \calT^0$. The linear span of the family $(\lambda_i)_{i\in \calT^0}$ is denoted $\Lambda^0(\calT)$.

For each integer $k \geq 1$ and $T\in \calT^k$, denote by $\lambda_T$ the compatible $k$-form:
\begin{equation} 
\lambda_T= k! \sum_{i=0}^{k} (-1)^i \lambda_{\sigma_T(i)} \rmd \lambda_{\sigma_T(0)} \wedge \cdots (\rmd \lambda_{\sigma_T(i)})^{\wedge} \cdots \wedge \rmd \lambda_{\sigma_T(k)},
\end{equation}
where the symbol $(\cdot)^{\wedge}$ signifies omission.
The linear span of the family  $(\lambda_T)_{T \in \calT^k}$ is denoted $\Lambda^k(\calT)$. It can alternatively be characterized as:
\begin{equation}
\myspan \{ \sum_{i=0}^{k} (-1)^i u_i \rmd u_0 \wedge \cdots  (\rmd u_i)^{\wedge} \cdots \wedge \rmd u_k \ : \ \forall i\in [k] \quad u_i \in \Lambda^0(\calT)\}.
\end{equation}
For any $k$ we denote by $\calP[k]$ the group of permutations of $[k]$, and by $\epsilon : \calP[k] \to \{-1, 1\}$ the signature morphism. We notice that for any functions $u_0, \cdots, u_k\in \Omega^0(\calT)$ we have:
\begin{equation}
k! \sum_{i=0}^{k} (-1)^i u_i \rmd u_0 \wedge \cdots (\rmd u_i)^{\wedge} \cdots \wedge \rmd u_k = \sum_{\tau \in \calP[k]} \epsilon(\tau) u_{\tau(0)} \rmd u_{\tau(1)} \wedge \cdots \wedge \rmd u_{\tau(k)}.
\end{equation}

We define the \emph{degree of freedom} associated with a $k$-simplex $T$ as the linear form $\mu_T$ on $\Omega^k(\calT)$ defined by :
\begin{equation}
\mu_T : u \mapsto \int_{|T|} u|_T,
\end{equation}
when $|T|$ is oriented by $\sigma_T$.

\begin{proposition}
For any two $k$-simplexes $T$ and $T'$, we have $\mu_{T'} (\lambda_{T})= \delta_{TT'}$, where the last symbol is the Kronecker delta.
\end{proposition}
\begin{proof}
First we remark that, since the family $(\lambda_i)_{i \in T}$ constitutes a partition of unity on $|T|$, we have:
\begin{equation}
(\lambda_T)|_T = k! \ (\rmd \lambda_{\sigma_T(1)} \wedge \cdots \wedge \rmd \lambda_{\sigma_T(k)})|_T.
\end{equation}
The integral of this $k$-form on $|T|$ can be computed in coordinates and is 1 by the choice of orientation and normalization.

Next we remark that if  $T' \neq T$ we can pick an $i \in T \setminus T'$. The function $\lambda_i$ is $0$ on $|T'|$ (as well as its exterior derivative), implying that $\lambda_{T}|_{T'}=0$. This completes the proof.
\end{proof}

Since this proposition implies linear independence of the family $(\lambda_T)_{T \in \calT^k}$ we obtain:
\begin{corollary}
The family $(\lambda_T)_{T \in \calT^k}$ is a basis for $\Lambda^k(\calT)$.
\end{corollary}

For any $(k+1)$-simplex $T$ and any $k$-face $T'$ of $T$ we define the \emph{incidence number} $\epsilon(T,T')\in \{-1, 1 \}$ to be the sign of the permutation of $[k+1]$ defined as:
\begin{eqnarray}
  0 & \mapsto & \sigma_T^{-1}(0),\\
i+1 & \mapsto & \sigma_T^{-1}(\sigma_{T'}(i) + 1) \myfor i \in [k].
\end{eqnarray}
For all simplexes $T,\ T'$ not covered by this definition we put $\epsilon(T,T')=0$.

\begin{proposition}\label{prop:fe}
 The exterior derivative $\rmd$ maps $\Lambda^k(\calT)$ into $\Lambda^{k+1}(\calT)$ and the matrix of $\rmd : \Lambda^k(\calT) \to \Lambda^{k+1}(\calT)$ in the bases $(\lambda_{T'})_{T'\in \calT^k} \to (\lambda_T)_{T\in \calT^{k+1}}$ has entry $\epsilon(T,T')$ at the indices $(T,T')\in \calT^{k+1} \times \calT^k$.
\end{proposition}
\begin{proof}
First we remark that:
\begin{equation}
\rmd \lambda_{T'} = (k+1)!\ \rmd \lambda_{\sigma_{T'}(0)} \wedge \cdots \wedge \rmd \lambda_{\sigma_{T'}(k)}.
\end{equation}
Next we make the following computation. Let $(u_i)_{i\in I}$ denote any family of functions in $\Omega^0(\calT)$ constituting a partition of unity of $|\calT|$ indexed by a finite set $I$ containing $[k]$ as a subset. Choose a $k<n$. We have:
\begin{eqnarray*}
& &\rmd u_0 \wedge \cdots \wedge \rmd u_k\\
&=&\sum_{j\not \in [k]} u_j \rmd u_0 \wedge \cdots \wedge \rmd u_k + \sum_{i \in [k]} u_i \rmd u_0 \wedge \cdots \wedge \rmd u_k\\
&=&\sum_{j\not \in [k]} u_j \rmd u_0 \wedge \cdots \wedge \rmd u_k + \sum_{i \in [k]} u_i \rmd u_0 \wedge \cdots \big(\rmd (1- \sum_{j \not \in [k]} u_j)\big)_{\mathrm{at}\ i} \cdots \wedge \rmd u_k\\
&=&\sum_{j\not \in [k]} \Big( u_j \rmd u_0 \wedge \cdots \wedge \rmd u_k + \sum_{i \in [k]} (-1)^{i+1} u_i \rmd u_j \wedge \rmd u_0 \wedge \cdots (\rmd u_i)^\wedge \cdots \wedge \rmd u_k \Big)
\end{eqnarray*}
If we change the index tuple $(j,0,1, \cdots ,k)$ to $(0,1, \cdots, k+1)$ and remark that for any family $(u_i)_{i\in [k+1]}$ of elements of $\Omega^0(\calT)$ we have:
\begin{eqnarray*}
& &u_0 \rmd u_1 \wedge \cdots \wedge \rmd u_{k+1} + \sum_{i \in [k]} (-1)^{i+1} u_{i+1} \rmd u_0 \wedge \cdots (\rmd u_{i+1})^\wedge \cdots \wedge \rmd u_{k+1}\\
&=&\sum_{i \in [k+1]} (-1)^{i} u_{i} \rmd u_0 \wedge \cdots (\rmd u_{i})^\wedge \cdots \wedge \rmd u_{k+1},
\end{eqnarray*}
then, comparing the signs appearing here with the definition of incidence number gives:
\begin{equation}\label{eq:dlambda}
\rmd \lambda_{T'} = \sum_{T\in \calT^{k+1}}\epsilon(T,T')\lambda_T.
\end{equation}
This completes the proof.
\end{proof}

\begin{remark}\label{rem:order}
Suppose that the set $\calT^0$ has been given a total ordering and that for each $T\in \calT^k$ one chooses the bijection $\sigma_T: [k]\to T$ to be the one which is increasing. Pick $T\in \calT^{k+1}$ and a $k$-face $T'$ of $T$. Let $l\in [k+1]$ denote the integer such that:
\begin{equation}
\{\sigma_T(l)\}= T \setminus T'. 
\end{equation}
 Then:
\begin{equation}
\epsilon(T,T')= (-1)^l.
\end{equation}
\end{remark}

We thus have a diagram :
\begin{equation}\label{eq:firstdiagram}
\begin{array}{ccccccccc}
0 & \to & \Omega^0(\calT) & \to  &\Omega^1(\calT) & \to & \Omega^2(\calT) & \to & \cdots\\
\uparrow& &\uparrow& &\uparrow& &\uparrow& & \\
0 & \to & \Lambda^0(\calT) & \to  &\Lambda^1(\calT) & \to & \Lambda^2(\calT) & \to & \cdots
\end{array}
\end{equation}
where the horizontal arrows represent exterior derivatives and the vertical arrows are inclusion mappings. We refer the reader to Appendix \ref{sec:compl} for some notions of homological algebra which will be applied to the above diagram.  Since $\rmd \circ \rmd = 0$, each row is a \emph{complex}. The vertical arrows of course commute with the exterior derivative, providing a \emph{morphism of complexes}. The \emph{cohomology group} $\rmH^k \Omega^\bullet(\calT)$ is defined as the quotient of the kernel of $\rmd:\Omega^k(\calT) \to\Omega^{k+1}(\calT)$ by the range of $\rmd:\Omega^{k-1}(\calT) \to \Omega^k(\calT)$. The cohomology of $\Lambda^\bullet (\calT)$ is defined similarly, and one checks that morphisms of complexes induce morphisms between cohomology groups.

\begin{remark} By  Proposition \ref{prop:fe} and Remark \ref{rem:order}, it follows that the complex $\Lambda^\bullet(\calT)$ is isomorphic to the simplicial cochain complex defined for instance in Gelfand-Manin \cite{GelMan03} \S I.4.1. Said differently, the simplicial cochain complex is the coordinate expression of $\Lambda^\bullet(\calT)$ in the bases $(\lambda_T)_{T\in \calT^k}$ when the simplexes are oriented in a certain way.\end{remark}

So-called interpolators $\Pi_\calT^k :\Omega^k(\calT) \to \Lambda^k(\calT)$ are constructed by assigning to $u\in\Omega^k(\calT)$ the unique element $v\in \Lambda^k(\calT)$ such that for each $T\in \calT^k$:
\begin{equation}\label{eq:interpol}
\mu_T v = \mu_T u.
\end{equation}
Evidently, these are projections. Moreover we have:
\begin{proposition}
The interpolators $\Pi_\calT^\bullet$ commute with the exterior derivative, i.e. we have commuting diagrams:
\begin{equation}
\begin{array}{ccc}
\Omega^k(\calT)& \witharrowunder{\rmd} & \Omega^{k+1}(\calT)\\
\witharrowright{\Pi_\calT^k}& & \witharrowright{\Pi_\calT^{k+1}}\\
\Lambda^k(\calT)& \witharrowunder{\rmd} & \Lambda^{k+1}(\calT)
\end{array}
\end{equation}
\end{proposition}
\begin{proof}
Application of Stokes theorem in a simplex $T$, taking into account relative orientations, gives:
\begin{eqnarray}
\int_T \rmd \lambda_{T'} &=& \sum_{T''} \epsilon(T, T'') \int_{T''}\lambda_{T'},\\
& = & \epsilon(T,T').
\end{eqnarray}
By Proposition \ref{prop:fe} (specifically equation (\ref{eq:dlambda})) this gives the desired result.
 \end{proof}

\section{Cohomology\label{sec:cohom}}

 We will now show that the inclusion mappings (vertical arrows) in (\ref{eq:firstdiagram}) induce \emph{isomorphisms} between cohomology groups. Then we will examine intermediate complexes appearing in finite element theory.

\subsection{An isomorphism in cohomology}

We examine first the properties of a single simplex. 

First we recall the construction of homotopy operators. Their efficiency in the context of finite elements was noticed by Hiptmair \cite{Hip99}.

Let $T$ be a $p$-simplex equipped with an affine realization $\rho_T : T \to V$. Let $x_0$ be one of its vertices and define functions $F_t : V \to V$ by:
\begin{equation}
F_t(x) = x_0 + t ( x - x_0),
\end{equation}
Define also a vector field $X$ on $V$ by:
\begin{equation}
X(x) = x -x_0.
\end{equation}
Introduce finally $A$, a map taking $k$-forms on $|T|$ to $k-1$-forms on $|T|$, defined by (here $v \lfloor X$ is the contraction of the form $v$ by the vectorfield $X$):
\begin{equation}
A v = \int_0^1 F_t^\star ( v \lfloor X) \rmd t.
\end{equation}
Then it is standard procedure to remark that for any function $u$ on $|T|$:
\begin{equation}
u - u(x_0) = A \rmd u,
\end{equation}
and if $u$ is a $k$-form on $|T|$ with $k \geq 1$:
\begin{equation}
u = A \rmd u + \rmd A u.
\end{equation}
As a consequence we obtain:
\begin{proposition}
We have $\dim \rmH^0 \Omega^\bullet (T) = 1$ and for $k \geq 1$, $\rmH^k \Omega^\bullet (T) = 0$.
\end{proposition}

As noticed by Hiptmair  (even for higher order polynomials), by the special choice of homotopy $F_t$, $A$ also maps $\Lambda^k(T)$ into $\Lambda^{k-1}(T)$, and it follows that $\dim \rmH^0 \Lambda^\bullet (T) = 1$ and for $k \geq 1$, $\rmH^0 \Lambda^\bullet (T) = 0$. For the purposes of future reference we therefore state the trivial consequence:
\begin{proposition}
The inclusions $\Lambda^k(T) \to \Omega^k(T)$ induce isomorphisms\\
$\rmH^k \Lambda^\bullet (T) \to \rmH^k \Omega^\bullet (T)$.
\end{proposition}
\begin{proof}
Only the case $k=0$ needs to be considered, and it is straightforward.
\end{proof}

We will also need a similar result for the boundaries of simplexes. 
First we state the following:
\begin{proposition}\label{prop:extension}
The restriction map $\Omega^k(T) \to \Omega^k(\partial T)$ is onto, i.e. compatible piece-wise smooth $k$-forms on boundaries of simplexes can be extended to the interior.
\end{proposition}

\begin{proof}
For $i\in [k]$ put $x_i = \rho_T \sigma_T (i)$. In the following $T$ will often be identified with the set $\{x_0, \cdots, x_k\}$, enabling us for instance to call the points $x_i$ the vertices of $T$. Let $U$ be the affine subspace of $V$ generated by $|T|$ and $U_i$ be the complement in $U$ of the affine span of the face of $T$ opposite to $x_i$. Let $(\phi_i)$ denote a partition of unity of a neighborhood of $|T|$ in $U$, subordinated to the open covering $(U_i)$. Thus $\phi_i$ is $1$ on a neighborhood of $x_i$ and $0$ on a neighborhood of the affine realization of the face of $T$ opposite to $x_i$. 

Pick an $l$-face $T'$ of $T$ and denote by $T''$ the opposite face $T \setminus T'$. Let $x_{i_0}, \cdots, x_{i_l}$ denote the vertices of $T'$. In this setting we (temporarily) denote the barycentric coordinate on $T$ associated with the vertex $x_i$ by $\lambda_i$, and define $F_{T'}$ to be the map $|T|\setminus |T''| \to |T'|$ defined in barycentric coordinates by:
\begin{equation}
(\lambda_0, \cdots, \lambda_p) \mapsto (\lambda_{i_0}, \cdots, \lambda_{i_l})/(\lambda_{i_0}+ \cdots +\lambda_{i_l}).
\end{equation}
We also define $\phi_{T'}$ to be the function:
\begin{equation}
\phi_{T'}= \phi_{i_0}+ \cdots + \phi_{i_l}.
\end{equation}
Remark that $\phi_{T'}$ is $1$ on a neighborhood of $|T'|$ and $0$ on a neighborhood of $|T''|$.

Pick $u\in \Omega^k(\partial T)$. If $T'$ is a face of $T$ and the restriction of $u$ to $\partial T'$ is $0$ then the following $k$-form on $|T|$ :
\begin{equation}
E_{T'} u = \phi_{T'} F_{T'}^\star (u|_{T'}),
\end{equation}
satisfies :
\begin{equation}
(i_{TT'})^\star E_{T'} u = u|_{T'},
\end{equation}
and for any other face $T''$ of $T$ with the same dimension as $T'$ we have (since $u |_{\partial T'} =0 $):
\begin{equation}
(i_{TT''})^\star E_{T'} u = 0.
\end{equation}
It follows that if $u\in \Omega^k(\partial T)$ is $0$ on all geometric realizations of faces of dimension $\leq l$, then the restriction to $\partial T$ of the $k$-form on $|T|$:
\begin{equation}
\sum_{T'\in \calT^{l+1}} E_{T'} u,
\end{equation}
coincides with $u$ on all those with dimension $\leq l+1$.

Now define an extension of $u$ recursively by:
\begin{equation}
E_0 u = \sum_{T'\in \calT^0} E_{T'} u,
\end{equation}
and for $0\leq l\leq k-2$:
\begin{equation}
E_{l+1} u = \sum_{T'\in \calT^{l+1}} E_{T'} (u - (E_0 u + \cdots + E_l u)|_{\partial T}).
\end{equation}

Then by the above considerations $E_0 u + \cdots + E_{k-1} u$ is indeed an extension of $u$ to $|T|$.
\end{proof}

\begin{remark}
The similar result for $\Lambda^\bullet$ follows easily from the fact that the barycentric coordinate maps $\lambda_i$ can be extend naturally from the boundary to the simplex (yielding barycentric coordinate maps).
\end{remark}
The space of $k$-forms on $T$ whose restriction to the boundary $\partial T$ is $0$ is denoted $\Omega^k(T,\partial T)$. The exterior derivative commutes with restrictions yielding maps:
\begin{equation}
\rmd : \Omega^k(T,\partial T) \to \Omega^{k+1}(T,\partial T).
\end{equation}
Thus we obtain a subcomplex of $\Omega^\bullet(T)$ denoted $\Omega^\bullet(T, \partial T)$. The next proposition is about its cohomology groups (called relative cohomology groups).

\begin{proposition}
For any $k< \dim T$, $\rmH^k \Omega^\bullet(T, \partial T) = 0$ i.e. for any $u \in \Omega^k(T)$ such that $u|_{\partial T}=0$ and $\rmd u =0$, there is $v \in \Omega^{k-1}(T)$ such that $v|_{\partial T}=0$ and $\rmd v = u$. 

For $k= \dim T$ we have $\dim \rmH^k \Omega^\bullet(T, \partial T) = 1$ and, for  any $u \in \Omega^k(T)$ such that $u|_{\partial T}=0$ and $\rmd u =0$, there is $v \in \Omega^{k-1}(T)$ such that $v|_{\partial T}=0$ and $\rmd v = u$ if and only if\footnote{Integration being taken with respect to any chosen orientation of $T$.} $\int u =0$.
\end{proposition}
The following induction argument is a variant of the one used to prove a similar result for polynomial forms in \cite{Hip99} (Lemma 17, which however does not distinguish the case $k= \dim T$).

\begin{proof}
We proceed by induction on the dimension of $T$. For $\dim T =0$ the result is clear. Suppose now that the proposition is true for simplexes of dimension $ \leq j$ for some integer $j$, and consider a simplex $T$ of dimension $j+1$. 

First we suppose that $k< \dim T$. Pick $u \in \Omega^k(T)$ such that $u|_{\partial T}=0$ and $\rmd u =0$. We use the homotopy operators constructed above, relative to a vertex $x_0$ of $T$. Let $T'$ denote the face of $T$ opposite to $x_0$.
Put $v=Au\in \Omega^{k-1}(T)$. Then we have $\rmd v = u$ and for any proper face $T''$ of $T$ -- except perhaps $T'$ -- $v |_{T''}=0$ (since $T''$ is in a proper face of $T$ containing $x_0$). We correct for this eventuality using the induction hypothesis:

Put $v'=v|_{T'}$. Then $v'$ is a $(k-1)$-form on the $(\dim T -1)$-simplex $T'$, such that $v'|_{\partial T'}=0$ and $\rmd v' = u|_{T'}= 0$ . We can therefore pick $w'\in \Omega^{k-2}(T')$ such that $w'|_{\partial T'}=0 $ and $\rmd w' = v'$. Then extension of $w'$ by $0$ to $\partial T$ yields an element of $\Omega^{k-2}(\partial T)$ still denoted $w'$, and we can extend this to an element $w$ of $\Omega^{k-2}(T)$ by Proposition \ref{prop:extension}. Then $v - \rmd w$ satisfies $\rmd (v -\rmd w) = u$ and by construction $(v - \rmd w)|_{\partial T}=0$.

Now we treat the case $k= \dim T$. If $\rmd v = u$ and $v |_{\partial T}=0$, the integral of $u$ is necessarily $0$ by Stokes formula:
\begin{equation}
\int_{T} u = \int_{T} \rmd v = \int_{\partial T} v|_{\partial T} =0.
\end{equation}
In particular any constant non-zero element $\Omega^{k}(T)$ provides a non-zero relative cohomology class. On the other hand if $\rmd u =0$ and $\int u =0$ we can use a construction similar to the preceding case, noticing that, with (as above) $v=Au$ and $v'=v|_{T'}$:
\begin{equation}
\int_{T'}v'= \int_{\partial T}v|_{\partial T} = \int_T\rmd v= \int_T u =0.
\end{equation}
This completes the induction step and hence the proof.
\end{proof}

\begin{proposition}
The inclusions $\Lambda^k(T, \partial T) \to \Omega^k(T, \partial T)$ induce isomorphisms\\
$\rmH^k \Lambda^\bullet (T, \partial T) \to \rmH^k \Omega^\bullet (T,\partial T)$.
\end{proposition}
\begin{proof}
Pick $u\in \Lambda^k(T, \partial T)$. We remark the following:\\
-- If $k< \dim T$ then $u|_{\partial T} = 0$ implies $u=0$.\\
-- If $k=\dim T$ then $\int u = 0$ implies $u=0$.\\
From these two remarks and the preceding proposition the result follows.
\end{proof}

\begin{proposition}\label{prop:boundary}
The inclusions $\Lambda^k(\partial T) \to \Omega^k(\partial T)$ induce isomorphisms\\ $\rmH^k \Lambda^\bullet (\partial T) \to \rmH^k \Omega^\bullet (\partial T)$.
\end{proposition}
\begin{proof}
For each $k$ we have a short exact sequences of the type:
\begin{equation}
0 \to \Omega^k (T,\partial T) \to \Omega^k(T) \to \Omega^k (\partial T) \to 0.
\end{equation}
They provide a long exact sequence in cohomology (see appendix \ref{sec:compl}). We have a similar result for $\Lambda^\bullet$. Thus, for each $k$, we have a commutative diagram:
\begin{equation}\label{eq:commdiagram}
\begin{array}{ccccc}\!\!\!
\rmH^k \Omega^\bullet (T,\partial T) & \!\!\!\to  \rmH^k\Omega^\bullet(T)  \to \!\!\!&\rmH^k\Omega^\bullet (\partial T)&\!\!\! \to  \rmH^{k+1} \Omega^\bullet (T,\partial T)  \to \!\!\!& \rmH^{k+1}\Omega^\bullet(T)\\
\uparrow& \uparrow &\uparrow &\uparrow &\uparrow\\
\rmH^k \Lambda^\bullet (T,\partial T) & \!\!\!\to  \rmH^k\Lambda^\bullet(T)  \to\!\!\! &\rmH^k\Lambda^\bullet (\partial T)& \!\!\!\to  \rmH^{k+1} \Lambda^\bullet (T,\partial T)  \to\!\!\! & \rmH^{k+1}\Lambda^\bullet(T)\\
\end{array}
\end{equation}
where each row is an exact sequence, part of the long exact sequence in cohomology obtained from the short ones, 
and the vertical arrows are maps induced by the inclusions of complexes $\Lambda^\bullet \to \Omega^\bullet$.

In diagram (\ref{eq:commdiagram}) the two first and the two last vertical arrows are isomorphisms. It then follows by the \emph{five lemma} (see Gelfand-Manin \cite{GelMan03} p. 120) that the middle arrow is an isomorphism. 
\end{proof}

%\subsection{Properties of complexes}
The preceding results on simplexes can be extended to simplicial complexes as follows:
\begin{theorem}\label{theo:cohom}
The inclusion mappings $\Lambda^k(\calT) \to \Omega^k(\calT)$ induce isomorphisms in cohomology $\rmH^k\Lambda^\bullet(\calT) \to \rmH^k\Omega^\bullet(\calT)$.
\end{theorem}
\begin{proof}
The simplicial complex $\calT$ can be constructed by starting  with a $0$-simplex and adjoining at each step a top-dimensional simplex whose boundary is already included. The case of a complex consisting only of a $0$-simplex has already been treated (if need be) so we consider now a simplicial complex $\calT'$ for which the theorem is true and a simplex $T \notin \calT' $ such that $\partial T \subset \calT'$ and the dimension of $T$ is larger than, or equal to, the largest dimension of the elements of $\cal T'$. Let $\calT$ denote the simplicial complex $\calT' \cup \{T\}$.

We consider the short exact sequences of Mayer-Vietoris type (see \cite{Spa66} chapter 4, \S 6):
\begin{eqnarray}
0 \to\Omega^k(\calT) \to \Omega^k(\calT')\times \Omega^k(T) \to \Omega^k(\partial T) \to 0,
\end{eqnarray}
where the second arrow is the injection:
\begin{equation}
u \mapsto (u|_{\calT'}, u|_T),
\end{equation}
and the third arrow is the surjection (surjectivity follows from Proposition \ref{prop:extension}):
\begin{equation}
 (u, v) \mapsto u|_{\partial T} - v|_{\partial T}. 
\end{equation}
These provide a long exact sequence in cohomology.

We also have similar short exact sequences for $\Lambda^\bullet$ providing a long exact sequence in cohomology. The inclusion maps $\Lambda^\bullet\to \Omega^\bullet$ then provide a commuting diagram relating these two long exact sequences, much as in  Proposition \ref{prop:boundary}:
\begin{equation}
\begin{array}{ccccc}
\rmH^k \Omega^\bullet (\calT') \times \rmH^k \Omega^\bullet (T)& \!\!\! \to  \rmH^k\Omega^\bullet(\partial T)  \to \!\!\!&\rmH^{k+1}\Omega^\bullet (\calT)& \!\!\! \to  \rmH^{k+1} \Omega^\bullet (\calT') \times \rmH^{k+1} \Omega^\bullet (T) \to \!\!\! & \rmH^{k+1}\Omega^\bullet(\partial T)\\
\uparrow& \uparrow &\uparrow &\uparrow &\uparrow\\
\rmH^k \Lambda^\bullet (\calT') \times \rmH^k \Lambda^\bullet (T)& \!\!\! \to  \rmH^k\Lambda^\bullet(\partial T)  \to\!\!\! &\rmH^{k+1}\Lambda^\bullet (\calT)& \!\!\! \to  \rmH^{k+1} \Lambda^\bullet (\calT') \times \rmH^{k+1}\Lambda^\bullet (T) \to\!\!\! & \rmH^{k+1}\Lambda^\bullet(\partial T)\\
\end{array}
\end{equation}
Once again the five lemma then gives the desired result.
\end{proof}

\subsection{On high order finite element spaces\label{sec:fe}}

Consider the following situation: suppose we are given spaces $\Gamma^k(\calT)$ such that $\Lambda^k(\calT) \subset \Gamma^k(\calT) \subset \Omega^k(\calT)$ and $\rmd$ maps $\Gamma^k(\calT)$ into $\Gamma^{k+1}(\calT)$. We suppose furthermore that the spaces $\Gamma^k(\calT)$ are equipped with interpolation operators $\Pi_\Gamma :\Omega^k(\calT) \to \Gamma^k(\calT)$ i.e. projectors with range $\Gamma^k(\calT)$, and that the interpolators commute with the exterior derivative. We also suppose that we have equipped $\Lambda^k(\calT)$ with interpolation operators $\Pi_\Lambda :\Omega^k(\calT) \to \Lambda^k(\calT)$ commuting with the exterior derivative.

\begin{proposition}\label{prop:highorder}
Suppose that $\Pi_\Lambda \circ \Pi_\Gamma = \Pi_\Lambda$. Then the inclusions $\Lambda^k(\calT) \to \Gamma^k(\calT)$ induce isomorphisms in cohomology.
\end{proposition}
\begin{proof}
Since $\Pi_\Lambda$ is a left inverse of the inclusion map $\Lambda^\bullet(\calT) \to \Omega^\bullet(\calT)$ and the latter induces isomorphisms in cohomology, $\Pi_\Lambda$ also induces isomorphisms in cohomology. Thus the \emph{kernel complex} (see appendix \ref{sec:compl}) of $\Pi_\Lambda$ is exact.

Consider the morphism of complexes $\Gamma^\bullet(\calT) \to \Lambda^\bullet(\calT)$ obtained by restricting $\Pi_\Lambda$ ; we shall prove that its kernel complex is exact, from which it follows that it induces isomorphisms in cohomology.

Suppose $u \in \Gamma^k(\calT)$ satisfies $\Pi_\Lambda u = 0$ and $\rmd u =0$. We can pick a $v \in \Omega^{k-1}(\calT)$ such that $\Pi_\Lambda v = 0$ and $\rmd v = u$. Then $\Pi_\Gamma v$ is an element of $\Gamma^{k-1}(\calT)$ such that $\Pi_\Lambda \Pi_\Gamma v = \Pi_\Lambda v = 0$ and $ \rmd \Pi_\Gamma v= \Pi_\Gamma u=u$.

This completes the proof. 
\end{proof}

For lowest order finite element spaces we have already constructed interpolators determined by equation (\ref{eq:interpol}). For high or variable order finite elements the construction of suitable interpolators is more technical. Usually they are constructed by imposing conditions similar to (\ref{eq:interpol}), for a given choice of simplexes $T$ with $\dim T \geq k$ and where $v$ and $u$ are perhaps differentiated and then wedged with an appropriate space of polynomial forms before being integrated; one says that one imposes high order moments to match. Interpolators of this type, commuting with the exterior derivative, were constructed in Demkowicz et al. \cite{DemMonVarRac00}, see also Hiptmair \cite{Hip02} \S 3.5. In any case we remark that most often, to construct interpolators for high order finite elements, one imposes matching moments with respect to a family of linear forms on $\Omega^k(\calT)$ containing $(\mu_T)_{T \in \calT^k}$ as a subfamily, and this guarantees that the hypothesis $\Pi_\Lambda \circ \Pi_\Gamma = \Pi_\Lambda$ holds.

One way of constructing higher order finite element spaces is the following.
Put $X^k_1 = \Lambda^k(\calT)$ for each $k$. Define $X^0_n$ to be the vectorspace generated by functions of the form $u_1 u_2 \cdots u_n$ with $u_i \in X^0_n$ for each $i$. Then define, for $k \geq 1$ and $n\geq 2$, $X^k_n$ to be the vectorspace generated by compatible $k$-forms of the form $uv$ with $u \in X^0_{n-1}$ and $v \in X^k_1$. Our main result on  these spaces is:
\begin{proposition}\label{prop:wedge}
For each $k,l$ and $m,n$ the wedge of forms provides a map:
\begin{equation}
\wedge : X^k_m \times X^l_n \to X^{k+l}_{m+n}
\end{equation}
\end{proposition}
\begin{proof}
We prove that $\wedge : X^k_1 \times X^l_1 \to X^{k+l}_{2}$, from which the proposition follows immediately.

We will need some notations:

Let $(u_i)$ be some family of functions indexed by consecutive integers. For any set of consecutive integers $k < \cdots < l$ we put:
\begin{equation}
\delta u_{[k \cdots l]}= \rmd u_k \wedge \cdots \wedge \rmd u_l.
\end{equation}
This notation will also be used when one index is missing in the set $\{k, \cdots, l \}$, e.g.:
\begin{equation}
\delta u_{[k \cdots \hat{i} \cdots l]}= \rmd u_k \wedge \cdots (\rmd u_i)^\wedge \cdots \wedge \rmd u_l.
\end{equation}
We also put:
\begin{equation}
u_{[k \cdots l]}= (-1)^k\sum_{i=k}^l (-1)^i u_i \delta u_{[k \cdots \hat{i} \cdots l]}, 
\end{equation}
a notation which is extended straightforwardly to the case of one missing index.

We will prove, by induction on $k$, that:
\begin{equation}
u_{[0 \cdots k-1]} \wedge u_{[k \cdots k+l]} = (-1)^{k-1} \sum_{i=0}^{k-1} (-1)^i u_i u_{[0 \cdots \hat{i} \cdots k+l]}.
\end{equation}
It is evidently true for $k = 1$ and, if it is true for a given $k \geq 1$, we can make the following computations. We remark that:
\begin{equation}
u_{[-1 \cdots k-1]} \wedge u_{[k \cdots k+l]} = (u_{-1} \delta u_{[0 \cdots k-1]} - \delta u_{-1} \wedge  u_{[0 \cdots k-1]})\wedge u_{[k \cdots k+l]}.
\end{equation}
Concerning the first term on the right hand side, we see that:
\begin{equation}\label{eq:firstpart}
u_{-1} \delta u_{[0 \cdots k-1]}\wedge u_{[k \cdots l]} = (-1)^k u_{-1}(u_{[0 \cdots k+l]} - \sum_{i=0}^{k-1}u_i \delta u_{[0 \cdots \hat{i} \cdots k+l]}).
\end{equation}
For the second term we remark that (by the induction hypothesis):
\begin{eqnarray}
& & \delta u_{-1} \wedge  u_{[0 \cdots k-1]}\wedge u_{[k \cdots k+l]}\\
 &=& (-1)^{k-1}\delta u_{-1} \wedge \sum_{i=0}^{k-1} (-1)^i u_i u_{[0 \cdots \hat{i} \cdots k+l]}\\
\label{eq:secondpart}
&=& (-1)^{k-1}\sum_{i=0}^{k-1} (-1)^i u_i (-u_{[-1 \cdots \hat{i} \cdots k+l]} + u_{-1} \delta u_{[0 \cdots \hat{i} \cdots k+l]})
\end{eqnarray}
Now the last term in (\ref{eq:firstpart}) cancels with the last term in (\ref{eq:secondpart}) and we are left with:
\begin{equation}
u_{[-1 \cdots k-1]} \wedge u_{[k \cdots k+l]} = (-1)^k ( u_{-1} u_{[0 \cdots k+l]} - \sum_{i=0}^{k-1} (-1)^i u_i u_{[-1 \cdots \hat{i} \cdots k+l]}).
\end{equation}
This completes the induction and hence the proof.
\end{proof}

\section{Regularization\label{sec:regul}}

Unfortunately the interpolation operators such as $\Pi_\calT^\bullet$ do not have sufficient continuity properties with respect to Sobolev norms, a problem which becomes more acute in higher dimensions. To remedy on this we construct regularization operators and show some applications to discrete Poincar\'e-Friedrichs inequalities and discrete Rellich compactness.

\subsection{Smooth manifolds}
Let $\calT$ be a simplicial complex and $V$ an affine space. Suppose we have a map $\rho : \calT^0 \to V$ and define for each $T \in \calT$ a map $\rho_T : T \to V$ by $\rho_T(i)= \rho(i)$; it is an affine realization of $T$ iff the affine span of $\rho_T(T)$ has dimension the dimension of $T$. We suppose now that this is the case for each $T$. Then we have an affine realization of $\calT$. We say that $\rho$ is non-degenerate if in addition, for each $T,T' \in \calT$:
\begin{equation}
|T| \cap |T'| = | T \cap T'|.
\end{equation}
In these circumstance we also say that $\rho$ determines an embedding of $|\calT|$ in $V$.

Consider a Riemannian manifold $M$, which for simplicity we suppose to be smooth, compact and without boundary. The dimension of $M$ is denoted $n$. The Riemannian metric gives rise to natural scalar products of forms on $M$.  

We suppose that we have built a simplicial complex $\calT$, with an affine realization, and for each $T\in \calT$ an embedding $\Xi_T:|T| \to M$. We require that for each $T,T'\in \calT$ if $T' \subset T$ we have:
\begin{equation}
\Xi_{T'}= \Xi_{T} \circ i_{TT'}.
\end{equation}
We suppose furthermore that for all $T,T'\in \calT$ the following condition on overlaps holds:
\begin{equation}
\Xi_{T}(|T|) \cap \Xi_{T'}(|T'|) = \Xi_{T''}(|T\cap T'|).
\end{equation}
Finally we suppose that $M$ is generated by $\calT$, i.e.:
\begin{equation}
M = \cup_{T \in \calT} \Xi_{T}(|T|).
\end{equation}

In fact we consider a countable family $(\calT_h)$ of simplicial complexes, indexed by a parameter $h$ representing the mesh width of $\calT_h$. This parameter runs trough a set of positive reals accumulating only at $0$, and  we are interested in the limit $h \to 0$. For notational simplicity we suppose that these simplicial complexes are two by two disjoint (the only reason is to assure that a given simplex ``remembers'' which simplicial complex it belongs to).

Given a complex $\Gamma^\bullet(\calT_h)$ of compatible differential forms, as in the preceding section, one can transport it to a complex of forms on $M$ denoted $\Gamma^\bullet_h(M)$, by pulling back by the maps $\Xi_T^{-1}$ for $T \in \calT_h$. These transported forms are $\rmL^2(M)$ (with respect to the Riemannian metric) and their exterior derivative in the sense of distributions coincides with the transported exterior derivative we defined on the affine realization of $\calT_h$, and is also $\rmL^2(M)$.

As is standard we denote by $\Omega^k(M)$ the space of smooth $k$-forms on $M$, whereas the space of piecewise smooth ones which are compatible with $\calT_h$ is denoted $\Omega^k_h(M)$.

\subsection{Construction of regularizations}
% In addition, for the case involving a family of simlicial complexes indexed by a parameter $h$,  we can obtain norm estimates which are well controlled with respect to $h$.

To fix notations we recall the construction of regularization by convolution. Let $\varphi$ denote a smooth non-negative function on $\bbR^n$ with support in the open unit ball, whose integral is $1$. For $\delta >0$ define $\varphi^\delta$ by :
\begin{equation}
\varphi^\delta (x) = \delta^{-n} \varphi(\delta^{-1} x).
\end{equation}
Thus $\varphi^\delta$ has integral $1$ and support in the open ball centered at the origin and with radius $\delta$.
 We let $R^\delta$  denote convolution by  $\varphi $, acting on differential forms on $\bbR^n$. That is:
\begin{equation}
(R^\delta u)(x) = \int \varphi^\delta (y) u(x-y) \rmd y,
\end{equation}

This regularization will be transferred to $M$ in three steps.
In accordance with widespread usage we will call the star of an element of $\calT_h^n$ (the star was defined in equation (\ref{eq:star})), a \emph{macroelement}. We suppose that there is a finite family $(S_i)_{i \in I}$ of $n$-dimensional macroelements embedded in $\bbR^n$ such that for each $h$ and each $T \in \calT^n_h$ we can choose a piecewise affine homeomorphism:
\begin{equation}
\Phi_T : |S_i| \to | \st (T)|.
\end{equation}
These will be called reference macroelements. For each $i\in I$ we denote by $T_i$ the central simplex of $S_i$ so that $S_i = \st(T_i)$. Let $\psi_i$ denote a smooth non-negative function on $\bbR^n$ with support in the interior of $|S_i|$ yet equal to $1$ on a neighborhood of $|T_i|$.

First we define regularization in the reference macroelements.  Define a regularization operator $R^\delta_i$ on $|S_i|$ by:
\begin{equation}\label{eq:regrefmacro}
R^\delta_i u =  (1- \psi_i)u + R^\delta (\psi_i u).
\end{equation}
This operator regularizes uniformly on $|T_i|$ and does not deregularize anywhere. We will always suppose that $\delta$ is so small that the $\delta$ neighborhood of $\supp\psi_i$ has adherence $A_i$ in the interior of $|S_i|$. Then $R^\delta_i u -u$ has support in $A_i$ for all $u$. 

Then we define regularization on the macroelements of $\calT_h$. For any $T\in \calT_h^n$, define $R^\delta_T$ by:
\begin{equation}
R^\delta_T u = (\Phi_T^{-1})^\star R^\delta_i (\Phi_T)^\star u.
\end{equation}
This notation is somewhat abusive, but $\Phi_T (A_i)$ is in the interior of $|\st(T)|$ and $(\Phi_T^{-1})^\star R^\delta_i (\Phi_T)^\star$ as an operator $\rmL^2(|\st(T)|) \to \rmL^2(|\st(T)|)$ does not modify its argument outside $\Phi_T (A_i)$; we can therefore extend it outside  $\Phi_T (A_i)$ by the identity operator.

Finally we order $\calT^n_h$, so its elements can be written $T_1,\ T_2, \cdots, T_N$ and define regularization $R^\delta_h$ on $M$ by:
\begin{equation}\label{eq:regprod}
R^\delta_h =  R^\delta_{T_N} \cdots R^\delta_{T_2} R^\delta_{T_1}.
\end{equation}
(The operator $R^\delta_h$ depends on the chosen order).

\subsection{Properties of the regularizations}

\paragraph{Approximation results.} 
For the rest of this section we consider a fixed (arbitrary) $k$. Unless explicitly stated otherwise we deal here with differential forms of degree $k$. We put $X_h  = \Gamma^k_h(M)$.

A key lemma is the following local stability result:
\begin{lemma}
For each $\delta$ there is a constant $C$ such that for all $h$, all $T \in \calT_h^n$ and all $u\in \rmL^2(\st(T))$ we have:
\begin{equation}
\|\Pi_h R^\delta_h u \|_{\rmL^2(T)} \leq C \| u\|_{\rmL^2(\st(T))}.
\end{equation} 
\end{lemma}
\begin{proof}
If we pull back to the reference macroelement $S_i$ assigned to $T$ by the map $\Phi_T$, and consider first the action of $R^\delta_h$ we see that we have several contributions corresponding to the different terms in the product (\ref{eq:regprod}). We have the action of $R^\delta_i$, which smoothes out on a neighborhood of $|T_i|$. But we also have the action of $R^\delta_j$ coming from the reference macroelements $S_j$ corresponding the neighboring simplexes of $T$ in $\calT_h^n$. Since the transition maps between reference macroelements are continuous piecewise affine, these actions do not destroy the property of being piecewise smooth and compatible on a neighborhood of $|T_i|$. Moreover the number of different possibilities for different macroelements to thus act upon each other is finite. For each integer $r$ we therefore have estimates of the form:
\begin{equation}
\| R^\delta_h u\|_{\rmC^r(|T_i|)} \leq C \|u\|_{\rmL^2(\st(T_i))}
\end{equation}
Then we use the continuity of $\Pi_h : \rmC^r(|T_i|) \to \rmL^2(|T_i|)$ and transport the estimate back to $M$.
\end{proof}

For each $\delta$ this lemma immediately implies uniform boundedness of the maps $\Pi_h R^\delta_h: \rmL^2(M) \to \rmL^2(M)$. But we actually have more:
\begin{theorem}\label{theo:L2conv}
For each $\delta$ and for each $u\in \rmL^2(M)$ we have:
\begin{equation}
\lim_{h \to 0} \| u - \Pi_h R^\delta_h u \|_{\rmL^2(M)} = 0,
\end{equation}
\end{theorem}
\begin{proof}
Since the maps  $\Pi_h R^\delta_h : \rmL^2(M) \to \rmL^2(M)$ are uniformly bounded it is enough to prove the above estimate for a dense subset of $\rmL^2(M)$.

If $u$ is a pullback from an open subset $U \subset \bbR^n$ to $M$ of a differential form which is the product of an alternating multilinear map on $\bbR^n$ by the characteristic function of a compact subset of $U$ with nice enough boundary (say Lipschitz), then the norm convergence follows from the remark that  one can ensure convergence in  the reference macroelement for constant functions, and bound the contribution of the macroelements touching the boundary of $\supp u$, by the preceding local lemma.
\end{proof}

When $u$ is in a space compactly embedded in $\rmL^2(M)$ we can obtain uniform convergence estimates, as follows.
\begin{proposition}\label{prop:compnormconv}
For each Banach space $X$ which can be considered as a subspace of $\rmL^2(M)$ such that the canonical injection $X \to \rmL^ 2(M)$ is compact we have, for each $\delta$:
\begin{equation}
\lim_{h\to 0} \| I - \Pi_h R^\delta_h\|_{X \to \rmL^2(M)} = 0.
\end{equation}
\end{proposition}
\begin{proof}
This follows from the preceding theorem, using Lemma \ref{lem:autom} in the appendices, letting  $(A_n)$ be the sequence $(\Pi_h R^\delta_h)$ indexed by $h$ for fixed $\delta$, and $K$ be the injection $X \to \rmL^2$.
\end{proof} 

The Galerkin spaces spaces $X_h$ are not invariant under $\Pi_h R^\delta_h$ -- but almost:
\begin{proposition}\label{prop:pseudostable}
For each $\epsilon$ there is $\delta'$ such that for each  $\delta < \delta'$ we have, for all $h$ and all $u_h \in X_h$: 
\begin{equation}
\| u_h - \Pi_h R^\delta_h u_h \|_{\rmL^2(M)} \leq \epsilon \| u_h \|_{\rmL^2(M)}.
\end{equation}
\end{proposition}
\begin{proof}
For each reference macroelement the space of pullbacks of elements of $X_h$ for all $h$, is included in a finite dimensional space. We can therefore find $\delta'$ such that for all $i\in I$, all $\delta < \delta'$, all $h$ and all pullbacks $u_h$ to $S_i$ of elements of $X_h$ we have: 
\begin{equation}
\| u_h - \Pi_h R^\delta_h u_h \|_{\rmL^2(|T_i|)} \leq \epsilon \| u_h \|_{\rmL^2(|S_i|)}.
\end{equation}
 This estimate is then transported back to $M$.
\end{proof}

\paragraph{Commutator estimates.} The operator $\Pi_h R^\delta_h$ does not commute with the exterior derivative but the following results provide some remedies. They show that the commutator:
\begin{equation}
[\Pi_h R^\delta_h , \rmd] = \Pi_h R^\delta_h \rmd - \rmd \Pi_h R^\delta_h,
\end{equation}
enjoys properties similar to $I -\Pi_hR^\delta_h$. It is worth recalling here that $\Pi_h$ commutes with $\rmd$, as do pullback operations.

As before we first obtain a local stability result:
\begin{lemma}
For each $\delta$ there is a constant $C$ such that for all $h$, all $T \in \calT_h^n$ and all $u\in \rmL^2(\st(T))$ we have:
\begin{equation}
\|\Pi_h [R^\delta_h, \rmd] u \|_{\rmL^2(T)} \leq C \| u\|_{\rmL^2(\st(T))}.
\end{equation} 
\end{lemma}
\begin{proof}
The following computation in the reference macroelement $S_i$ is useful:
\begin{eqnarray}
R^\delta_i \rmd u - \rmd R^\delta_i u & = & (1- \psi_i)\rmd u + R^\delta (\psi_i \rmd u) - \rmd ((1- \psi_i)u) - \rmd R^\delta (\psi_i u),\\
& = & (I - R^\delta)((\rmd \psi_i) \wedge u).
\end{eqnarray} 
It shows that $u$ is not differentiated. Remark also that $(\rmd \psi_i) \wedge u$ is equal to zero on a neighborhood of $|T_i|$. 

Next we remark that if we restrict attention to what happens on (a small neighborhood of) the simplex $T \in \calT^n_h$, the definition of $R^\delta_h$ can be taken to be:
\begin{equation}
R^\delta_h= R^\delta_{T_N} \cdots R^\delta_{T_2} R^\delta_{T_1},
\end{equation}
where \emph{only} the simplexes $T_m$ neighboring $T$ are used (i.e. $T_m \cap T \neq \emptyset$). Then we can write:
\begin{equation}
[R^\delta_h, \rmd] = \sum_{m=1}^N R^\delta_{T_N} \cdots R^\delta_{T_{m+1}}[R^\delta_{T_m}, \rmd]R^\delta_{T_{m-1}} \cdots R^\delta_{T_1}.
\end{equation}
In this sum the term where $T_m =T$ gives zero on a neighborhood of $|T|$, whereas the others constitute a small number (independent of $T$ and $h$) of products each containing the regularization $R^\delta_T$, composed (left and right) with terms which do not deregularize.

From this the lemma follows.
\end{proof}

The estimates for $\Pi_h [R^\delta_h , \rmd]$ then build up exactly as for $I -\Pi_hR^\delta_h$.

\begin{theorem}\label{theo:L2convcommut}
For each $\delta$ and for each $u\in \rmL^2(M)$ we have:
\begin{equation}
\lim_{h \to 0} \|\Pi_h [R^\delta_h, \rmd] u \|_{\rmL^2(M)} = 0,
\end{equation}
\end{theorem}

\begin{proposition}\label{prop:compcommut}
For each Banach space $X$ which can be considered as a subspace of $\rmL^2(M)$ such that the canonical injection $X \to \rmL^ 2(M)$ is compact we have, for each $\delta$:
\begin{equation}
\lim_{h\to 0} \| \Pi_h [R^\delta_h, \rmd] \|_{X \to \rmL^2(M)} = 0.
\end{equation}
\end{proposition}

\begin{proposition}\label{prop:pseudostablecommut}
For each $\epsilon$ there is $\delta'$ such that for each  $\delta < \delta'$ we have, for all $h$ and all $u_h \in X_h$: 
\begin{equation}
\|\Pi_h [R^\delta_h, \rmd] u_h \|_{\rmL^2(M)} \leq \epsilon \| u_h \|_{\rmL^2(M)}.
\end{equation}
\end{proposition}

\subsection{Applications}
While the properties of the operator $\Pi_h R^\delta_h$ may seem weak -- it does not commute with the exterior derivative, nor does it leave the Galerkin space invariant -- we show that they are strong  enough to prove some central estimates in numerical analysis.
Specifically we are interested in estimates on discrete analogues of Hodge decompositions.

For any $k$ the space of harmonic $k$-forms is the space:
\begin{equation}
\calH^k(M) = \{u \in \Omega^k(M) \ : \ \rmd u = 0 \myand \forall v \in \Omega^{k-1}(M) \quad \int u \cdot \rmd v = 0 \}.
\end{equation}
On the other hand the space of discrete harmonic $k$-forms is:
\begin{equation}
\calH^k_h(M) = \{u \in \Gamma^k_h(M) \ : \ \rmd u = 0 \myand \forall v \in \Gamma^{k-1}_h(M) \quad \int u \cdot \rmd v = 0 \}.
\end{equation}

For the rest of this section we consider a fixed (arbitrary) $k$, and put $X_h  = \Gamma^k_h(M)$. We will use also use the following notations:
\begin{eqnarray}
W_h & = & \{u_h \in X_h \ : \ \rmd u_h = 0 \}\\
V_h & = & \{u_h \in X_h \ : \ \forall w_h \in W_h \quad \ts \int u_h \cdot w_h = 0 \}
\end{eqnarray}

Similarly we let $X$ be the completion of $\Omega^k(M)$ with respect to the norm defined by: 
\begin{equation}
\|u\|_X^2 =  \| u\|_{\rmL^2}^2 + \| \rmd u\|_{\rmL^2}^2.
\end{equation}
Then we put:
\begin{eqnarray}
W & = & \{u \in X \ : \ \rmd u = 0 \}\\
V & = & \{u \in X \ : \ \forall w \in W \quad \ts \int u \cdot w = 0 \}
\end{eqnarray}
It will be useful to let $P_V$ denote the projection in $X$ with range $V$ and kernel $W$.

The continuous Poincar\'e-Friedrichs inequality is the assertion that there is a $C>0$ such that:
\begin{equation}
\forall v \in V \quad \|v\|_{\rmL^2} \leq C \| \rmd v \|_{\rmL^2}.
\end{equation} 
Since the manifold $M$ is supposed to be compact (without boundary) $V$ is a subspace of the Sobolev space $\rmH^1(M)$. In particular $V$ is compactly embedded in $\rmL^2$, which is a variant of Rellich compactness. A major point is that in general $V_h$ is \emph{not} a subspace of $V$. However we will prove analogues of the Poincar\'e-Friedrich and Rellich properties for the family of spaces $(V_h)$.

The following is a discrete Poincar\'e-Friedrichs inequality. 
\begin{proposition}\label{prop:discfried}
There is $C > 0$ such that for all $h$ we have:
\begin{equation}
\forall v_h \in V_h \quad \|v_h\|_{\rmL^2} \leq C \| \rmd v_h \|_{\rmL^2}.
\end{equation} 
\end{proposition}
\begin{proof}
Pick $v_h \in V_h$. We remark that $\rmd P_V v_h = \rmd v_h$. Let then:
\begin{equation}\label{eq:hodgedec}
v_h - P_V v_h = \rmd u + f,
\end{equation}
be the Hodge decomposition of $v_h -P_V v_h$ (see e.g. Taylor \cite{Tay96} Chapter 5, \S 8, Proposition 8.2). That is, $f$ is harmonic and the only thing we need to know about $u$ is that its norm with respect to a space compactly embedded in $\rmL^2$, is controlled by the $\rmL^2$ norm of $v_h - P_V v_h$, hence also by the $\rmL^2$ norm of $v_h$.  Remark that $\rmd \Pi_h R^\delta_h u + \Pi_h f\in W_h$. In the following, all unspecified norms are $\rmL^2$ norms. We write:
\begin{eqnarray}
\| v_h\| & \leq & \|v_h - \rmd \Pi_h R^\delta_h u - \Pi_h f \|,\\
& \leq & \| P_V v_h + \rmd u + f - \rmd \Pi_h R^\delta_h u - \Pi_h f \|,\\
& \leq & \|P_V v_h + \rmd u - \Pi_h R^\delta_h \rmd u + \Pi_h R^\delta_h \rmd u - \rmd \Pi_h R^\delta_h u + f - \Pi_h f \|,\\
& \leq & \|P_V v_h\| + \| (I- \Pi_h R^\delta_h)\rmd u \| + \| \Pi_h [R^\delta_h, \rmd ] u \| + \| (I - \Pi_h)f \|.
\end{eqnarray}
We treat separately the three last terms on the right hand side.

1) Concerning the second term we remark that, by the definition (\ref{eq:hodgedec}):
\begin{equation}
\| (I- \Pi_h R^\delta_h)\rmd u \| \leq   \| (I- \Pi_h R^\delta_h)P_V v_h \| + \| (I- \Pi_h R^\delta_h) v_h \| + \| (I- \Pi_h R^\delta_h) f \|.
\end{equation}
In this estimate we have three terms. The first term is controlled by:
\begin{equation}
\| (I- \Pi_h R^\delta_h)P_V v_h \| \leq \| (I- \Pi_h R^\delta_h)\|_{V \to \rmL^2} \| P_V v_h\|_X,
\end{equation}
and we are ready to apply Proposition \ref{prop:compnormconv}. For the second term we choose now $\delta>0$ such that the estimate of Proposition \ref{prop:pseudostable} holds with $\epsilon = 1/2$.  The last term is controlled using Theorem \ref{theo:L2conv} and the finite-dimensionality of the space of harmonic forms. Combining these three estimates we get, for a certain sequence $\epsilon_h$ converging to $0$:
\begin{equation}
\| (I- \Pi_h R^\delta_h)\rmd u \| \leq  1/2 \|v_h\| + \epsilon_h (\|v_h\| + \| \rmd v_h \|).
\end{equation}

2) As already remarked, $u$ is controlled in a space $Y$ (e.g. the Sobolev space $\rmH^1(M)$) which is compactly embedded in $\rmL^2$.
For the third term we can write:
\begin{equation}
\| \Pi_h [R^\delta_h, \rmd ] u \| \leq \| \Pi_h [R^\delta_h, \rmd ]\|_{Y \to \rmL^2}C\|v_h\|_{\rmL^2},
\end{equation}
and are ready to apply Proposition \ref{prop:compcommut}.

3) For the fourth term we use the smoothness of harmonic forms, and the finite-dimensionality of the space they form. 

Combining all three estimates we conclude that there is a sequence $(\epsilon'_h)$ converging to $0$ such that:
\begin{equation}
1/2 \| v_h\| \leq \|P_V v_h\| + \epsilon'_h (\|v_h\| + \| \rmd v_h \|) \leq (C + \epsilon'_h)\|\rmd v_h\| + \epsilon'_h \| v_h\|. 
\end{equation}
This proves the proposition.
\end{proof}

An easy consequence which will be useful later is the following:
\begin{corollary}\label{cor:inffried}
There is $C > 0$ such that for all $h$ we have:
\begin{equation}
\forall u_h \in X_h \quad \inf_{w_h \in W_h}\|u_h - w_h\|_{\rmL^2} \leq C \| \rmd u_h \|_{\rmL^2}.
\end{equation} 
\end{corollary}

The discrete Poincar\'e-Friedrich estimate also provides useful so-called Inf-Sup conditions for saddle-point problems, in the sense of Brezzi \cite{Bre74}. Consider for instance the following construction: A Fortin operator $\Phi_h: X \to X_h$ can be constructed by assigning to $u\in X$ the unique element $u_h \in X_h$, solution of:
\begin{equation}
\mixed{u_h}{p_h}{X_h}{\rmd X_h}{u_h'}{p_h'}{\int u_h \cdot u_h' + \int p_h \cdot \rmd u_h' = \int u \cdot u_h'}{\int p_h'\cdot \rmd u_h= \int p_h' \cdot \rmd u}
\end{equation}
The announced Inf-Sup condition is:
\begin{corollary}
There is $C > 0$ such that for all $h$ we have:
\begin{equation}
\inf_{p_h \in \rmd X_h}\sup_{u_h \in X_h} \frac{\int p_h\cdot \rmd u_h}{\|p_h\|_{\rmL^2} \, \|u_h\|_{\rmL^2}} \geq 1/C.
\end{equation} 
\end{corollary}
This guarantees the stability (uniform boundedness) of the above Fortin operators $\Phi_h : X \to X$.
Notice that if $u \in \calH^k(M)$ then $\rmd \Phi_h u = 0$ and by density:
\begin{equation}
\forall v \in \Omega^{k-1}_h(M ) \quad \int u \cdot \rmd v = 0,
\end{equation}
so that $\Phi_h$ maps $\calH^k(M)$ into $\calH^k_h(M)$. The Fortin operator can therefore be used to obtain estimates on discrete harmonic forms. We will use the following. We have:
\begin{equation}
\| (I -\Phi_h)|_{\calH^k(M)}\|_{X \to X} \to 0.
\end{equation}
In particular $\Phi_h : \calH^k(M) \to \calH^k_h(M)$ is eventually injective. Using the cohomological properties of the preceding section we know that $\calH^k(M)$ and $\calH^k_h(M)$ have the same dimension, so that $\Phi_h : \calH^k(M) \to \calH^k_h(M)$ is invertible. Denote by $\Psi_h : \calH^k_h(M) \to \calH^k(M)$ its inverse. The following elementary estimate will be used:
\begin{equation}\label{eq:harmonicapprox}
\|I-\Psi_h\| \leq \|I-\Phi_h\|/(1-\|I-\Phi_h\|) \to 0.
\end{equation}

The following is a strengthening of the discrete Poincar\'e-Friedrichs estimate:
\begin{proposition}\label{prop:gap} For each $\epsilon$ there is $h'$ such that for each $h < h'$ we have:
\begin{equation}
\forall v_h \in V_h \quad \| v_h - P_V v_h \|_{\rmL^2} \leq \epsilon \| \rmd v_h \|_{\rmL^2}.
\end{equation}  
\end{proposition}
\begin{proof}
In the following all unspecified norms are $\rmL^2(M)$ norms.
Remark that since $\rmd (v_h - P_V v_h)=0$ we have:
\begin{equation}
\rmd \Pi_h R^\delta_h( v_h -PV v_h) = \Pi_h [\rmd, R^\delta_h](v_h - P_V v_h).
\end{equation}
Pick $\epsilon >0$. From Propositions \ref{prop:compcommut}, \ref{prop:pseudostablecommut} and \ref{prop:discfried}  there is $\delta>0$ and $h'$ such that for $h <h'$:
\begin{equation}
\| \rmd \Pi_h R^\delta_h( v_h -P_V v_h)\| \leq \epsilon/(2C) \|\rmd v_h \|,
\end{equation}
where $C$ is the constant appearing in  Corollary \ref{cor:inffried}. This corollary provides $w_h \in W_h$ such that:
\begin{equation}
\| \rmd \Pi_h R^\delta_h( v_h -P_V v_h) - w_h\| \leq \epsilon/2 \|\rmd v_h \|.
\end{equation}

Actually, using Proposition \ref{prop:pseudostable}, we will suppose that $\delta$ was also chosen such that for all $h$ and all $u_h \in X_h$:
\begin{equation}
\|(I - \Pi_h R^\delta_h) u_h\| \leq \epsilon/2 \|u_h\|.
\end{equation}

We remark that $w_h$ is orthogonal both to $v_h$ and $P_V v_h$. Now we write:
\begin{eqnarray}
& & \| v_h -P_V v_h\|^2\\
& = & \ts \int (v_h - P_V v_h)\cdot (v_h - P_V v_h),\\
& \leq & \ts \int (v_h - P_V v_h)\cdot (v_h - P_V v_h - \Pi_h R^\delta_h(v_h -P_V v_h) + w_h) + \\
& &  \quad  \epsilon/2 \| v_h - P_V v_h\| \|\rmd v_h \| ,\\
& \leq  & \ts \int (v_h - P_V v_h)\cdot ( (I - \Pi_h R^\delta_h)(v_h  - P_V v_h)) +\\
& &   \quad \epsilon/2 \| v_h - P_V v_h\| \|\rmd v_h \|,\\
& \leq & \| v_h -P_V v_h\| ( \|(I - \Pi_h R^\delta_h) v_h\| + \| (I - \Pi_h R^\delta_h) P_V v_h\| + \epsilon/2 \| \rmd v_h\| ).
\end{eqnarray}
This gives:
\begin{eqnarray}
\| v_h -P_V v_h\| &\leq & \|(I - \Pi_h R^\delta_h) v_h\| + \| (I - \Pi_h R^\delta_h) P_V v_h\| + \epsilon/2 \| \rmd v_h\|,\\
& \leq &  \epsilon/2 (\|v_h\| +\| \rmd v_h\|)  + \|I - \Pi_h R^\delta_h\|_{V \to \rmL^2} \|P_V v_h\|_X.
\end{eqnarray}
Applying Proposition \ref{prop:compnormconv} we get, for $h$ small enough: 
\begin{equation}
\| v_h -P_V v_h\| \leq \epsilon (\| v_h\| + \| \rmd v_h\|).
\end{equation}
By Proposition \ref{prop:discfried}, this completes the proof.
\end{proof}

The following corollary is often referred to as discrete compactness:
\begin{corollary}\label{cor:disccomp}
Suppose $(u_h)$ is a sequence of elements $u_h \in X^k_h$ such that $\|u_h\|_{\rmL^2}$ and $\|\rmd u_h\|_{\rmL^2}$ are bounded. Suppose furthermore that:
\begin{equation}
\forall h \quad \forall u'_h \in X^{k-1}_h \quad \int u_h \cdot \rmd u'_h = 0.
\end{equation}
Then (out of every subsequence) one can extract a subsequence converging in $\rmL^2$.
\end{corollary}
\begin{proof}
Indeed $u_h$ can be decomposed as $u_h= v_h + w_h$ where $v_h \in V_h$ and $w_h \in W_h$. In fact $w_h$ is discrete harmonic.
Now $P_V v_h$ is bounded in $V$ which is compactly embedded in $\rmL^2$, so $(P_V v_h)$ is relatively compact in $\rmL^2$. Moreover $\Psi_h w_h$ is harmonic and is bounded in a finite-dimensional space. One then concludes using Proposition \ref{prop:gap} and estimate (\ref{eq:harmonicapprox}) respectively.
\end{proof}
 
\appendix

\section{Complexes and the long exact sequence\label{sec:compl}}
We now recall some definitions and facts from homological algebra, and refer the reader to Lang \cite{Lan03} and Gelfand-Manin \cite{GelMan03} for thorough expositions.  In this paper by a \emph{complex} we mean a sequence $A^\bullet$ of vectorspaces equipped with linear operators $d^k : A^k \to A^{k+1}$ called \emph{differentials} and satisfying, for each $k$, $ d^{k+1} \circ d^k = 0$. The \emph{cohomology group} $\rmH^k A^\bullet$ is (the vectorspace) defined by:
\begin{equation}
\rmH^{k} A^\bullet = (\ker d^k: A^k \to A^{k+1})/(\im d^{k-1}: A^{k-1} \to A^k)
\end{equation}
Most often the index $k$ in $d^k$ is dropped and the complex is represented by a diagram:
\begin{equation}
\cdots \to A^{k-1} \to A^k \to A^{k+1} \to \cdots,
\end{equation}
where it is implicit that the arrows represent instances of the differential $d$.

If $A^\bullet$ and $B^\bullet$ are two complexes, a \emph{morphism of complexes} $f^\bullet: A^\bullet \to B^\bullet$, is a sequence of linear operators $f^k : A^k \to B^k$ such that the following diagram commutes:
\begin{equation}
\begin{array}{rcr}
A^k & \to  & A^{k+1}\\
\witharrowright{f^k} & & \witharrowright{f^{k+1}} \\
B^k & \to & B^{k+1}
\end{array}
\end{equation}
A morphism of complexes $f^\bullet: A^\bullet \to B^\bullet$ induces a sequence or linear maps $\rmH^k f^\bullet : \rmH^k A^\bullet \to \rmH^k B^\bullet$. We have that $\rmH^k (f^\bullet \circ g^\bullet) = (\rmH^k f^\bullet) \circ (\rmH^k g^\bullet)$.

Given two morphisms of complexes $f^\bullet, g^\bullet : A^\bullet \to B^\bullet$, a \emph{homotopy operator} from $f^\bullet$ to $g^\bullet$ is a family $h^\bullet$ of linear operators $h^k : A^k \to B^{k-1}$ such that:
\begin{equation}
g^k - f^k = h^{k+1}d^k + d^{k-1} h^{k}.
\end{equation}
If $f^\bullet$ and $g^\bullet$ are homotopic in the sense that a homotopy operator exists from one to the other, they induce the \emph{same} maps $\rmH^k A^\bullet \to \rmH^k B^\bullet$.

One of the basic tools used in this paper is the following:
Suppose we are given three complexes $A^\bullet$, $B^\bullet$ and  $C^\bullet$, and morphisms of complexes $f^\bullet: A^\bullet \to B^\bullet$ and $g^\bullet: B^\bullet \to C^\bullet$, providing for each $k$ a \emph{short exact sequence}:
\begin{equation}
0 \to A^k \to B^k \to C^k \to 0.
\end{equation}
Then (see e.g. Lang \cite{Lan03} chapter XX, \S 2) one can construct a \emph{long exact sequence} linking the cohomology groups:
\begin{equation}
\cdots \ \rmH^k A^\bullet \to \rmH^k B^\bullet  \to \rmH^k C^\bullet \to \rmH^{k+1} A^\bullet \to \rmH^{k+1} B^\bullet \cdots.
\end{equation}
Here the first arrow is $\rmH^k f^\bullet$, the second is $\rmH^k g^\bullet$ and the third one -- usually denoted $\delta^k$ -- is constructed by the \emph{snake lemma}.

\begin{remark}
In particular, if we are given a morphism of complexes $g^\bullet: B^\bullet \to C^\bullet$ consisting of surjections and define $A^k$ to be the kernel of $g^k$, we remark that $A^\bullet$ is a subcomplex of $B^\bullet$ called the \emph{kernel complex}, which has trivial cohomology if and only if $g^\bullet$ induces isomorphisms in cohomology $\rmH^k B^\bullet  \to \rmH^k C^\bullet$.
\end{remark}

Suppose that the above situation holds for  complexes $A^\bullet_0$, $B^\bullet_0$ and  $C^\bullet_0$ (with $f^\bullet_0$ and $g^\bullet_0$) and also for $A^\bullet_1$, $B^\bullet_1$ and  $C^\bullet_1$ (with $f^\bullet_1$ and $g^\bullet_1$). Suppose furthermore that we have morphisms of complexes $\alpha^\bullet: A_0^\bullet \to A_1^\bullet$, $\beta^\bullet: B_0^\bullet \to B_1^\bullet$ and $\gamma^\bullet: C_0^\bullet \to C_1^\bullet$ intertwining $f^\bullet_1$ with $f^\bullet_0$ and $g^\bullet_1$ with $g^\bullet_0$. Explicitly we have commuting diagrams:
\begin{equation}
\begin{array}{ccccccccc}
0& \witharrowunder{} & A^k_0 & \witharrowunder{f_0^k}  & B^k_0 & \witharrowunder{g_0^k} &  C^k_0 & \witharrowunder{}& 0\\
                     & & \witharrowright{\alpha^k} & & \witharrowright{\beta^k} & & \witharrowright{\gamma^k}\\
0& \witharrowunder{} & A^k_1 & \witharrowunder{f_1^k}  & B^k_1 & \witharrowunder{g_1^k} &  C^k_1 & \witharrowunder{}& 0
\end{array}
\end{equation}
Then the following diagram commutes:
\begin{equation}
\begin{array}{ccccccc}
  \rmH^k A^\bullet_0 & \witharrowunder{\rmH^k f^\bullet_0} & \rmH^k B^\bullet_0 & \witharrowunder{\rmH^k g^\bullet_0} & \rmH^k C^\bullet_0 & \witharrowunder{\delta^k_0} & \rmH^{k+1} A^\bullet_0 \\
 \witharrowright{\rmH^k\alpha^\bullet} & & \witharrowright{\rmH^k \beta^\bullet} & & \witharrowright{\rmH^k \gamma^\bullet} & & \witharrowright{\rmH^{k+1} \alpha^\bullet}\\
  \rmH^k A^\bullet_1 & \witharrowunder{\rmH^k f^\bullet_1} & \rmH^k B^\bullet_1 & \witharrowunder{\rmH^k g^\bullet_1} & \rmH^k C^\bullet_1 & \witharrowunder{\delta^k_1} & \rmH^{k+1} A^\bullet_1
\end{array}
\end{equation}

\section{Composition with compact operators}

\begin{lemma}\label{lem:autom}
Let $(A_n)$ be a sequence of continuous operators from a Hilbert space $Y$ to a Banach space $Z$, such that, for some operator $A$:
\begin{equation}
\forall u \in X \quad \lim_{n \to \infty} A_n u = Au.
\end{equation}
Then $A$ is continuous and for any Banach space $X$ and any compact operator $K:X \to Y$ we have:
\begin{equation}
\lim_{n \to \infty} \|AK - A_nK\|_{X \to Z} = 0.
\end{equation}
\end{lemma}
\begin{proof}
Continuity of $A$ follows from the uniform boundedness principle, and the norm convergence follows from the fact that $K$ can be approximated in the $\rmL(X,Y)$ norm by finite rank operators, see e.g. Rudin \cite{Rud91} (Chapter 4, exercise 13).
\end{proof}

\section*{Acknowledgments}
I am grateful to Geir Ellingsrud and Jon Rognes for helping me with the homological algebra; in particular they they showed me how to use long exact sequences and the five lemma to prove Theorem \ref{theo:cohom}. I also thank Robert Kotiuga for making me aware of Weil's paper \cite{Wei52}.

\end{document}